\documentclass[12pt]{amsart}
\usepackage{verbatim}
\usepackage{graphicx}
\usepackage{xypic}

\def\RR{\mathbb{R}}
\def\CC{\mathbb{C}}
\def\NN{\mathbb{N}}

\def\fin{\mathrm{fin}}
\def\z{\mathcal{Z}}
\def\x{\mathcal{X}}

\def\A{A}
\def\P{\mathcal{P}}

\newtheorem{lemma}{Lemma}
\newtheorem{thm}{Theorem}
\newtheorem*{citethm}{Theorem}
\newtheorem{cor}{Corollary}
\theoremstyle{definition}
\newtheorem*{remark}{Remark}

\DeclareMathOperator{\vol}{vol} \DeclareMathOperator{\pos}{Pos}

\begin{document}
\title{On the real multidimensional \\ rational $K$-moment problem}
\author{Jaka Cimpri\v c}
\address{\noindent
Cimpri\v c, Jaka,
Faculty of Mathematics and Physics, \newline \indent
University of Ljubljana,
Jadranska 21, SI-1000 Ljubljana, Slovenija, \newline \indent
email: cimpric@fmf.uni-lj.si}
\author{Murray Marshall}
\address{
Marshall, Murray,
Department of Mathematics and Statistics,\newline \indent
University of Saskatchewan,
Saskatoon, SK S7N 5E6, Canada,\newline \indent
email: marshall@math.usask.ca}
\author{Tim Netzer}
\address{
Netzer, Tim,
Fachbereich Mathematik und Informatik, \newline \indent
Universit\"{a}t Leipzig,
D- 04009 Leipzig,
Germany, \newline \indent
e-mail tim.netzer@math.uni-leipzig.de}

\date{July 30, 2008 and, in revised form, October 9th, 2009.}

\begin{abstract}
We present a solution to the real multidimensional rational
$K$-moment problem, where $K$ is defined by finitely many
polynomial inequalities. More precisely, let $S$ be a finite set
of real polynomials in $\underline{X}=(X_1,\ldots,X_n)$ such that
the corresponding basic closed semialgebraic set $K_S$ is
nonempty. Let $E=D^{-1}\RR[\underline{X}]$ be a localization of
the real polynomial algebra,  and $T_S^E$ the preordering on $E$
generated by $S$. We show that every linear functional $L$ on $E$
such that $L(T_S^E) \ge 0$ is represented by a positive measure
$\mu$ on a certain subset of $K_S$, provided $D$ contains an
element that grows fast enough on $K_S$.
\end{abstract}

\maketitle

\section{Introduction}

The \textit{moment problem} for a commutative unital $\RR$-algebra
$E$ asks to characterize real positive linear functionals on $E$
which can be represented as integrals over measures on an appropriate
representation space of $E$. If supports of the measures
are required to lie in a prescribed subset $K$ of the representation
space then we talk about the $K$-\textit{moment problem} on $E$.

A solution to the $K$-moment problem on
$\RR[\underline{X}]=\RR[X_1,\ldots,X_n]$ for a compact basic
closed semialgebraic set $K$ was given  by K. Schm\"udgen in
\cite{kon}. The aim of this paper is to extend his result, both in
the compact and non-compact case, to localizations of the
polynomial algebra, i.e. to algebras of the form
$E=D^{-1}\RR[\underline{X}]$, where $D$ is a multiplicative set.
The case $n=1$, $D$ generated by $X_1-\alpha$ for countably many
real $\alpha$, and $K$ a compact basic closed semialgebraic set
has already been done by J. D. Chandler in \cite{jdc}. 
Several papers deal with the case $n=1$, $D$ generated by $X_1-\alpha$ for
countably many real $\alpha$ and $K=\RR$, see \cite{bu,bu2} for surveys. 
In this case the
existence of the solution is rather trivial and the emphasis is on
the uniqueness of solutions. The multidimensional case with  $D$
generated by $1+\sum_{i=1}^n X_i^2+\sum_{i=1}^m g_i^2$ where
$g_1\ge 0,\ldots,g_m \ge 0$ are the defining relations of $K$ has
been done by M. Putinar and F.-H. Vasilescu in \cite{pv}.

To describe our main results we need some terminology. For a
finite subset $S=\{g_1,\ldots,g_m\}$ of $\RR[\underline{X}]$ write
$K_S=\{a \in \RR^n \mid g_1(a)\ge 0,\ldots,g_m(a)\ge 0\}$ and
$T_S^E$ for the set of all finite sums of elements
of the form $e^2 g_1^{\nu_1} \cdots g_m^{\nu_m}$ where 
$e \in E=D^{-1}\RR[\underline{X}]$ and $\nu_1,\ldots,\nu_m \in \{0,1\}$.
 We say that a rational function $R \in E$ is $\succeq
0$ on a set $\mathcal{X} \subseteq \RR^n$, if there exist $f \in
\RR[\underline{X}]$ and $d \in D$ such that $R=\frac{f}{d}$ and
$fd \ge 0$ on $\x$. We write $\mathcal{Z}(d)$ for the set of real zeros of a
polynomial $d$.

Our main results can be summarized as follows:

\begin{citethm}
Let $S$ be a finite subset of $\RR[\underline{X}]$, $D$ a
multiplicative subset of $\RR[\underline{X}]$ such that $1 \in D$, $0 \not\in D$, and let $R$ be an
element of $E = D^{-1} \RR[\underline{X}]$.
Assume there exists an element $p \in D$
such that $p\ge 1$ on $K_S$ and $kp\ge \sum_{i=1}^n X_i^2$ on $K_S$
for some integer $k \ge 1$. Then
\begin{enumerate}
\item $R$ belongs to the closure of $T_S^E$ in the finest locally
convex topology on $E$ if and only if $R \succeq 0$ on $K_S
\setminus \bigcup_{d \in D} \z(d)$,
\item $R$ belongs to the
closure of $T_S^E$ in the topology of finitely open sets on $E$ if
and only if $R \succeq 0$ on $K_S$.
\item For every linear functional $L$ on $E$ such that $L(T_S^E) \ge 0$, there
exists a measure $\mu$ on $\overline{K_S \setminus \bigcup_{d \in D} \z(d)}$ such that
\[
L\left(\frac{f}{d}\right) = \int \frac{f}{d} \ d \mu
\quad
\text{for every}
\quad
\frac{f}{d} \in E.
\]
\end{enumerate}
\end{citethm}

See Theorems \ref{main}, \ref{mainmom} and \ref{main2}. These results carry over to localizations of an arbitrary finitely generated $\RR$-algebra; see Theorem \ref{mainfg}. Assertion (3) of the theorem solves the moment problem on $E$. For the proof of assertion (3) we use assertion (1) together with a certain
rational version of the Riesz-Haviland Theorem which is a
generalization of \cite[Theorem 3.2.2]{mar} and \cite[Theorem
3.2]{v}; see Theorem \ref{rht}.

The conditions on the polynomial $p$ (that $p\ge 1$ on $K_S$ and $kp\ge \sum_{i=1}^n X_i^2$ on $K_S$
for some integer $k \ge 1$) are paraphrased by saying that ``$p$ grows fast enough on $K_S$''. See also  \cite[Theorem 6.2.3]{mar} and \cite[Theorem 5.1]{sch}. One can always take $p=1+\sum_{i=1}^n X_i^2$ (provided, of course, that $1+\sum_{i=1}^n X_i^2 \in D$).
If $K_S$ is compact one can take $p=1$, so, in this case, the only assumptions on the
multiplicative set $D$ are the trivial ones $1\in D, 0\notin D$.

\section{Preorderings and the finest locally convex topology}

\subsection*{Definition}
Let $E$ be an $\RR$-vector space. A  set $U\subseteq E$ is
called \textit{absorbent}, if for every $x\in E$ there exists
$\lambda> 0$ such that $x\in\lambda U.$ $U$ is called
\textit{symmetric}, if $\lambda U\subseteq U$ for all
$\lambda\in [-1,1].$

The set of all convex, absorbent and symmetric subsets of $E$ forms
a zero neighborhood base of a vector space topology on $E$ (see
\cite[II.25]{d}). It is called the \textit{finest locally convex
topology} on $E$ and the collection of all open sets is denoted by
$\mathcal{T}_{\omega}$.
\medskip

The following are well-known:
\begin{enumerate}
\item[(T1)] $(E,\mathcal{T}_{\omega})$ is a topological vector space,
i.e. addition and scalar multiplication are continuous.
Moreover,  $(E,\mathcal{T}_{\omega})$ is Hausdorff.
\item[(T2)] Every linear mapping from $(E,\mathcal{T}_{\omega})$ to
a vector space with any locally convex topology is continuous.
In particular, all functionals on $(E,\mathcal{T}_{\omega})$
are continuous.
\end{enumerate}
These properties are proved in \cite[II.26]{d} .

\medskip

A subset $T$ of a commutative ring $E$ with $1$ is a
\textit{preordering} if it is closed under addition and
multiplication and if it contains the set $E^2=\{a^2 \mid a \in
E\}$. As usual, $\sum E^2$ denotes the preordering of $E$ consisting of sums of squares.

\begin{lemma}
\label{l} For any preordering $T$ on any commutative $\RR$-algebra
$E$, its closure in $\mathcal{T}_{\omega}$ is also a preordering.
\end{lemma}

\begin{proof}
Since $T \subseteq \overline{T}$, $E^2 \subseteq \overline{T}$.
Since the addition on $E$ is continuous, it follows that
$\overline{T}$ is closed under addition. We can't use the same
argument for multiplication because by \cite[p. 734, Theorem
2]{w}, the multiplication on $E=\CC(X)$ is not continuous.
However, for every $f \in E$, the mapping $\phi_f \colon E \to E$,
$\phi_f(h)=fh$, is linear, hence continuous by (T2). So, for every
$f \in T$,
\[
\phi_f(\overline{T})\subseteq \overline{\phi_f(T)}\subseteq \overline{T}
\]
holds, as $\phi_f(T)\subseteq T.$ It follows that for every $g \in \overline{T}$
\[
\phi_g(\overline{T})\subseteq\overline{\phi_g(T)}\subseteq \overline{T},
\]
as $\phi_g(T)\subseteq \overline{T}$, by the above consideration.
So for any $g, h \in\overline{T}$, $g \cdot h=\phi_g(h)\in\overline{T}$.
\end{proof}

\begin{remark}
Note that every preordering $T$ on an $\RR$-algebra $E$ is a
convex cone, so by (T2) and the Separation Theorem for convex sets
(e.g. \cite[II.39, Corollary 5]{d}), an element $x$ belongs to
$\overline{T}$ if and only if $L(x) \ge 0$ for every linear
functional $L$ on $E$ such that $L(T) \ge 0$. It follows that for
all elements $a,b \in E$ such that $a + \epsilon b \in T$ for
every $\epsilon>0$, we have that $a \in \overline{T}$.
\end{remark}

From now on, we are mostly interested in localizations of the real
polynomial algebra $\RR[\underline{X}]$ in $n$ variables
$\underline{X}=(X_1,\ldots,X_n)$. So for a multiplicative set
$D\subseteq \RR[\underline{X}]\setminus\{0\}$ containing $1$, we
examine $E=D^{-1}\RR[\underline{X}]$.

\begin{remark}
While the vector space dimension of $\RR[\underline{X}]$ is
countable, the dimension of $E$ can be uncountable. For example,
if $D$ is the set of polynomials without real zeros, then the
uncountable family $ \left\{ \frac{1}{X_1^2 +c^2}\mid c\in
\RR^{\neq 0}\right\}\subseteq D^{-1}\RR[\underline{X}] $ is
linearly independent.
\end{remark}

For a set $\x\subseteq\RR^n$, an arbitrary function $\phi\colon\x\rightarrow\RR$
and an element $R \in E=D^{-1}\RR[\underline{X}]$, we say
\[
R\succeq\phi  \mbox{ on } \x,
\]
if there exist $f\in\RR[\underline{X}]$ and $d\in D$ such that
$R=\frac{f}{d}$ and $(f-d\phi) d \geq 0$ on $\x$. This just means
that $R$ has a representation $\frac{f}{d}$, $f\in \RR[\underline{X}]$, $d\in D$, where the function
$\frac{f}{d}$ is $\geq$ $\phi$ pointwise on $\x$, wherever it is
defined. The set
\[
\pos^E(\x)=\left\{R\in E\mid R\succeq 0 \mbox{ on } \x\right\}
\]
is a preordering in $E$. We also write
\[
\pos(\x)=\left\{f\in \RR[\underline{X}]\mid f\geq 0 \mbox{ on } \x\right\}
\]
for the preordering of all polynomials nonnegative on $\x$ in the usual sense.
Note that $-1 \not\in \pos(\x)$ unless $\x$ is empty. On the other hand,
$-1 \in \pos^E(\x)$ if $\x \subseteq \z(d)$ for some $d \in D$. However, if
$\x \setminus \z(d)$ is dense in $\x$ for every $d \in D$, such phenomena cannot occur.

\begin{lemma}
\label{posi}
For every element $R \in E=D^{-1}\RR[\underline{X}]$ and every
set $\x \subseteq \RR^n$ such that $\x \setminus \z(d)$ is dense
in $\x$ for every $d \in D$, the following are equivalent:
\begin{enumerate}
\item $R \in \pos^E(\x)$,
\item for every representation $R=\frac{f}{d}$ with $f \in \RR[\underline{X}]$
and $d \in D$, we have that $fd \in \pos(\x)$,
\item if we consider $R$ as an element of $\RR(\underline{X})$ and write
$R=\frac{a}{b}$ where $a, b\in \RR[\underline{X}]$,
$b \ne 0$ and $\gcd(a,b)=1$, then $ab \in \pos(\x)$.
\end{enumerate}
In particular, for all $f \in \RR[\underline{X}]$,
$\frac{f}{1} \in \pos^E(\x)$ if and only if $f \in \pos(\x)$.
\end{lemma}

\begin{proof}
Suppose that $R=\frac{f}{d}=\frac{a}{b}$ where $a, b \in
\RR[\underline{X}]$, $b \ne 0$, satisfy $\gcd(a,b)=1$ and $f \in
\RR[\underline{X}]$, $d \in D$. Then there exists $u \in
\RR[\underline{X}] \setminus \{0\}$ such that $f=ua$ and $d=ub$.
Clearly, $fd=abu^2 \in \pos(\x)$ if and only if $ab \in \pos(\x
\setminus \z(bu))$. By the assumption on $\x$, this is equivalent
to $ab \in \pos(\x)$. So, (1), (2) and (3) are equivalent. The
last claim is a special case of (3).
\end{proof}

For a finite subset $S=\{g_1,\ldots,g_m\}$ of $\RR[\underline{X}]$ let
\[
K_S=\{a\in\RR^n\mid g_1(a)\geq 0,\ldots, g_m(a)\geq 0\}
\]
be the basic closed semialgebraic set defined by $S$.  Let
\[
T_S=  \{ \sum_{\nu \in \{0,1\}^m} s_\nu g_1^{\nu_1} \cdots g_m^{\nu_m}
\mid s_\nu \in \sum \RR[\underline{X}]^2\}
\]
and
\[
T_S^E = \{ \sum_{\nu \in \{0,1\}^m} s_\nu g_1^{\nu_1} \cdots g_m^{\nu_m}
\mid s_\nu \in \sum E^2 \}
\]
be the preorderings generated by $S$ in  $\RR[\underline{X}]$ and $E=D^{-1}\RR[\underline{X}]$,
respectively. We always assume $K_S\neq \emptyset$, so $-1\not\in T_S.$
If $K_S$ is not contained in any $\z(d)$, $d \in D$, then also $-1 \not\in T_S^E$.

The following is our main result:

\begin{thm}\label{main}
Let $D\subseteq \RR[\underline{X}]\setminus\{0\}$ be a
multiplicative set containing $1$ and $S$ a finite subset of
$\RR[\underline{X}].$ Suppose there is some $p\in D$ such that
$p\geq 1$ on $K_S$ and $kp-\sum_{i=1}^nX_i^2\geq 0$ on $K_S$ for
some $k\geq 1.$ Then
$$\overline{T_S^E}=\pos^E(K_S\setminus \bigcup_{d\in D}\mathcal{Z}(d))$$ holds in
$E=D^{-1}\RR[\underline{X}].$
\end{thm}

The set $\x = K_S\setminus \bigcup_{d\in D}\mathcal{Z}(d)$ satisfies
the assumptions of Lemma \ref{posi}, which gives two characterizations of
$\pos^E(\x)$. Another characterization is that $R \in \pos^E(\x)$ if and only if $\chi(R) \ge 0$
for every unital $\RR$-algebra homomorphism $\chi \colon E \to \RR$ such that $\chi(S) \ge 0$.

\medskip

We will give the proof of Theorem \ref{main} in Section \ref{proof}.
In Section \ref{moment}, we will show that Theorem \ref{main} implies
the solution of the real multidimensional rational $K$-moment problem.
In Section \ref{variant} we will prove a variant of Theorem \ref{main}
for the topology of finitely open sets.

\medskip

Note that in case $K_S$ is compact, we can always take $p=1$ in
the theorem. In the noncompact case, the polynomial
$p=1+\sum_{i=1}^nX_i^2$ can always be used in the application of
Theorem \ref{main}, as long as it belongs to the multiplicative
set $D$. We record an easy corollary of Theorem \ref{main}, in
which this is the case. Therefore let $\RR(\underline{X};\P)$ be
the algebra of real rational functions with (real) poles only in a
given set $\P\subseteq\RR^n.$ It is the localization of the
polynomial algebra with respect to the set $D$ of nonzero
polynomials with real zeros only in $\P$. The set of zeros of
elements from $D$ equals $\P$. The polynomial
$p=1+\sum_{i=1}^nX_i^2$ belongs to $D$.

\begin{cor}\label{rat} 
In $E=\RR(\underline{X};\P)=D^{-1}\RR[\underline{X}]$ we have for arbitrary finite
sets $S\subseteq\RR[\underline{X}]$
$$\overline{T_S^E}=\pos^E(K_S\setminus \P).$$ In particular
\label{c1} $$ \overline{\sum \RR(\underline{X};\P)^2} =
\pos^E(\RR^n \setminus \P).$$ If  $\P=\RR^n$, then
$\RR(\underline{X};\P)= \RR(\underline{X})$ and $\pos^E(\RR^n
\setminus \P)=\RR(\underline{X})$, so
\[
\overline{\sum \RR(\underline{X})^2} = \RR(\underline{X}).
\]
In particular, there is no positive linear functional on $\RR(\underline{X})$.
\end{cor}

The last claim of Corollary \ref{rat} also follows from \cite[9.7.29]{palm}.

\begin{remark}
Let $L/\RR$ be a proper field extension. We claim that there is no
nontrivial positive linear functional on $L$. If $L$ is not real,
then every element is a sum of squares, so there is clearly no
nontrivial positive linear functional. Otherwise, every element
$a$ of $L\setminus\RR$ is transcendental over $\RR$. So $\RR(a)$
is isomorphic to $\RR(X)$, and by Corollary \ref{c1} there is no
nontrivial positive linear functional on $\RR(a)$.
\end{remark}

\begin{remark}
When $K_S$ is compact and disjoint from $\bigcup_{d \in D}\z(d)$,
then we can prove Theorem \ref{main} by the usual analytic trick.
Namely, in this case W\" ormann's trick implies that $T_S^E$ is
archimedean. (Since W\" ormann's trick works only for finitely
generated $\RR$-algebras, we apply it first to algebras
$\RR[\underline{X}]_d$ and then use that
$E=D^{-1}\RR[\underline{X}]$ is equal to $\bigcup_{d \in D}
\RR[\underline{X}]_d$.) Then the Kadison-Dubois Representation
Theorem implies that the closure of $T_S^E$ is equal to
$\pos^E(K_S)$. We refer the reader to \cite{mar} for the statement
of the W\" ormann's trick and the Representation Theorem. We will
not pursue this idea further.
\end{remark}

\section{The proof of Theorem \ref{main}}
\label{proof}

The easy inclusion $\overline{T_S^E} \subseteq
\pos^E(K_S\setminus\bigcup_{d\in D}\z(d))$ follows from the
obvious fact that evaluation in a point from
$K_S\setminus\bigcup_{d\in D}\z(d)$ defines a linear functional
$L$ on $E$  such that $L(T_S^E) \ge 0$.

The proof of the difficult inclusion $\overline{T_S^E} \supseteq
\pos^E(K_S\setminus\bigcup_{d\in D}\z(d))$ will be split into
several lemmas.

\begin{lemma}\label{comp}
Let $D \subseteq \RR[\underline{X}] \setminus \{0\}$ be a
multiplicative set containing $1$, $K$ a compact set contained in
$\P=\bigcup_{d \in D} \z(d)$, and $c \in \RR^{> 0}$. Then every
neighbourhood of zero in $D^{-1} \RR[\underline{X}]$ (in the
finest locally convex topology) contains an element $R$ such that
$R \succeq c\cdot \chi_K$ globally, where $\chi_K$ is the
characteristic function of $K$.
\end{lemma}

\begin{proof}

Let $N$ be a neighborhood of zero in $D^{-1}\RR[\underline{X}]$,
without loss of generality assume that $N$ is convex and
absorbent. So for any $f\in D^{-1}\RR[\underline{X}]$ there is a
number $\delta(f)>0$, such that $\lambda f\in N$ for all
$\lambda\in[-\delta(f),\delta(f)].$

For every $a \in \bigcup_{d \in D} \z(d)$ we have $d_a(a)=0$ for
some $d_a\in D$, so that $d_a(x)=(\nabla d_a)(a) \cdot
(x-a)+o(\Vert x-a \Vert)$. Write $c_a=\Vert (\nabla d_a)(a)
\Vert+1$ and pick $\eta_a >0$ such that 
\begin{equation}\label{defeta} 
\vert d_a(x) \vert \le c_a \Vert x-a \Vert \mbox{ on } B(a,\eta_a).
\end{equation}

Let $W$ be a cube containing $K$ and let $B_n$ be the unit ball in
$\RR^n$. Define
\begin{equation}\label{defla}
 \lambda := \frac{1}{6} \left(\frac{\vol(B_n)}{c \,\vol(W) }\right)^{\frac{1}{n}}
\end{equation}
and
\begin{equation}\label{defr}
r_a:=\min\{ \lambda \, c_a^{-2} \,
\delta(d_a^{-2n})^{\frac{1}{n}}, \vol(W)^{\frac{1}{n}}, \eta_a \}.
\end{equation}
So $r_a>0$ and
\[
K \subseteq\bigcup_{a\in K} B(a,\frac{r_a}{3}),
\]
where $ B(a,\frac{r_a}{3})$ denotes the open ball of radius
$\frac{r_a}{3}$ around $a$. By Wiener's Covering Lemma,
\cite[Lemma 4.1.1]{kp}, there are $a_1,\ldots,a_t\in K$ such that
\[
K\subseteq \bigcup_{i=1}^{t} B(a_i,r_{a_i})
\]
and the $B(a_i,\frac{r_{a_i}}{3})$ are pairwise disjoint.

For every ball $B$ with center in $W$ and with radius less or equal than half the side length of $W$ we have
\[
\vol(B\cap W)\geq \frac{1}{2^n}\vol(B),
\]
since $\vol(B\cap W)$ is minimal if the center of $B$ lies in one of the corners of $W$.
In particular, this is true for $B=B(a_i,\frac{r_{a_i}}{3})$, $i=1,\ldots,t$, by (\ref{defr}).
Since $B(a_i,\frac{r_{a_i}}{3})$ are pairwise disjoint, it follows that
\[
\sum_{i=1}^t \frac{1}{2^n}\vol \left(B(a_i,\frac{r_{a_i}}{3})\right)\leq \vol(W).
\]
Combining this with the volume formula
\[
\vol\left(B(a_i,\frac{r_{a_i}}{3})\right)= \left( \frac{r_{a_i}}{3} \right)^n \vol(B_n),
\]
we get 
\begin{equation}\label{ineq}
\sum_{i=1}^{t} r_{a_i}^n \leq \frac{6^n
\vol(W)}{\vol(B_n)}=\frac{1}{\lambda^n c},
\end{equation} 
where the last equality follows from (\ref{defla}). Now define
\[
R:=\frac{1}{\lambda^n \sum_{i=1}^{t} r_{a_i}^n}\cdot\sum_{i=1}^{t} \frac{r_{a_i}^{2n} \, c_{a_i}^{2n}}{d_{a_i}^{2n}}= 
\sum_{i=1}^t \frac{r_{a_i}^n}{\sum_{j=1}^t r_{a_j}^n}\cdot  \frac{r_{a_i}^n c_{a_i}^{2n}}{\lambda^n d_{a_i}^{2n}}.
\]
Clearly $R \in D^{-1} \RR[\underline{X}]$ and, as $\frac{r_{a_i}^n
c_{a_i}^{2n}}{\lambda^n} \leq \delta(d_{a_i}^{-2n})$ by (\ref{defr}), each
$\frac{r_{a_i}^n  c_{a_i}^{2n}}{\lambda^n  d_{a_i}^{2n}}$ lies in
$N$. So $R$, as a convex combination of such elements, also lies
in $N$. By (\ref{defeta}) and (\ref{defr}), 
$\frac{r_{a_i}^{2n} c_{a_i}^{2n}}{d_{a_i}^{2n}}\succeq \chi_{B(a_i,r_{a_i})}$
globally, so one checks that
$$\sum_{i=1}^{t} \frac{r_{a_i}^{2n} \,
c_{a_i}^{2n}}{d_{a_i}^{2n}}\succeq \chi_K.$$ 
Therefore,  $R \succeq c\cdot\chi_K$ by (\ref{ineq}).
\end{proof}

\begin{lemma} \label{compa} Let $D\subseteq
\RR[\underline{X}]\setminus\{0\}$ be a multiplicative set
containing $1$ and $S\subseteq\RR[\underline{X}]$ finite. Assume
there is some $p\in D$ such that $p\geq 1$ on $K_S$, and
$kp-\sum_{i=1}^nX_i^2\geq 0$ on $K_S$ for some $k\geq 1$. Let $K$
be a compact subset of $\P=\bigcup_{d\in D}\mathcal{Z}(d).$ Then
every $f\in\RR[\underline{X}]$ which is nonnegative on
$K_S\setminus K$ belongs to $\overline{T_S^E}$ in
$E=D^{-1}\RR[\underline{X}].$
\end{lemma}

\begin{proof}
Take any neighborhood of zero $N$ in $E$. Let $m$ be the minimum
of $f$ on $K_S \cap K$. By Lemma \ref{comp}, there exists $R \in
N$ such that $R \succeq -m\cdot \chi_K$. So we find a
representation $f+R =\frac{u}{v}$ with $uv \geq 0$ on $K_S$. By
\cite[Theorem 5.1]{sch} \footnote{ Note that our assumption on $p$
implies that condition $(\ast)$ of \cite[Theorem 5.1]{sch} is
satisfied. If $p$ satisfies, instead of $p\geq 1$ on $K_S$, the
stronger assumption $p-1 \in T_S$ - e.g. when
$p=1+\sum_{i=1}^nX_i^2$ - then we can use \cite[Theorem 6.2.3]{mar}
(or \cite[Corollary 3.2]{mar2}), which is a slightly weaker version of
\cite[Theorem 5.1]{sch}. }, the polynomial $uv$ belongs to the closure of the
preordering generated by $S$ in $\RR[\underline{X},\frac1p],$ so
also to the closure of $T_S^E$. So $f+R= (\frac{1}{v})^2 uv\in
(f+N)\cap \overline{T_S^E}.$ So $f\in \overline{T_S^E},$ as $N$
was arbitrary.
\end{proof}

In the next Lemma we get rid of the compactness assumption:

\begin{lemma}\label{arb}
Lemma \ref{compa} also holds when $K=\P$.
\end{lemma}

\begin{proof}
For every $\epsilon>0$, $f +\epsilon p^{\deg(f)+1}$ is nonnegative
on $K_S\setminus B_{\epsilon}$ for a closed ball $B_{\epsilon}$
around $0$. Define
$$K_{\epsilon}:=\left\{a\in K_S\cap B_\epsilon \mid f(x)\leq
-\epsilon\right\}.$$

$K_{\epsilon}$ is a compact subset of  $\P=\bigcup_{d\in
D}\mathcal{Z}(d)$  and $f+\epsilon+\epsilon p^{\deg(f)+1}$ is
nonnegative on $K_S \setminus K_{\epsilon}$.  By Lemma
\ref{compa}, $f +\epsilon +\epsilon p^{\deg(f)+1}$ belongs to the
closure of $T_S^E$ in $E=D^{-1}\RR[\underline{X}]$. As this is
true for all $\epsilon >0$, $f$ belongs to $\overline{T_S^E}$.
\end{proof}

Now we give the

\begin{proof}[Proof of Theorem \ref{main}]
Take $\frac{f}{d} \in E$ with $fd \in \pos(K_S\setminus \P)$. Then
apply Lemma \ref{arb} to obtain $fd \in \overline{T_S^E}$ and
multiply with $(\frac{1}{d})^2.$
\end{proof}

\section{Rational moment problems}
\label{moment}

The aim of this section is to prove the following existence result
for the multidimensional rational $K$-moment problem, which in the
one-dimensional compact case extends \cite[Theorem 5]{jdc}.

\begin{thm}\label{mainmom}  Let $D\subseteq
\RR[\underline{X}]\setminus\{0\}$ be a multiplicative set
containing $1$ and let $S\subseteq\RR[\underline{X}]$ be finite.
Assume there is some $p\in D$ such that $p\geq 1$ on $K_S$, and
$kp-\sum_{i=1}^nX_i^2\geq 0$ on $K_S$ for some $k\geq 1$.

Then,  for every linear functional $L$ on
$E=D^{-1}\RR[\underline{X}]$ such that $L(T_S^E) \ge 0$, there
exists a measure $\mu$ on $\overline{K_S \setminus \bigcup_{d\in
D}\mathcal{Z}(d)}$ such that
\[
L\left(\frac{f}{d}\right) = \int \frac{f}{d} \ d \mu
\]
for every $\frac{f}{d} \in D^{-1}\RR[\underline{X}]$.
\end{thm}

If $K_S$ has empty intersection with $\P=\bigcup_{d \in D}\z(d)$,
then Theorem \ref{mainmom} follows from our Theorem \ref{main} and
\cite[Theorem 3.2.2]{mar}. To prove the general case, we need the
following generalization of \cite[Theorem 3.2.2]{mar} (applied to
$A=\RR[\underline{X}]$, $\x=\overline{K_S \setminus \bigcup_{d \in
D} \z(d)}$, $\hat{a}=a|_\x$ and $q=\sum X_i^2$.)

\begin{thm}[Rational Haviland's Theorem]\label{rht}
Let $A$ be a unital $\RR$ algebra, $D$ a multiplicative subset of
$A$ containing $1$, $\x$ a nonempty Hausdorff topological space
and $\ \hat{} \colon A \to \mathcal{C}(\x,\RR)$ an $\RR$-algebra
homomorphism. Suppose that:
\begin{enumerate}
\item there exists $q \in A$ such that $\hat{q} \ge 0$ on $\x$
and, for each $k \ge 1$, the set $\x_k=\{x \in \x \mid \hat{q}(x)
\le k\}$ is compact, \item for every $d \in D$, the zero set
$\z(\hat{d})$ of $\hat{d}$ in $\x$ has empty interior, \item every
open subset of $\x$ is $\sigma$-compact (= a countable union of
compact sets).
\end{enumerate}
Then, for any linear functional $L  \colon D^{-1} A \to \RR$,
satisfying
\[
\forall a \in A, \forall d \in D, \ \hat{a} \hat{d} \ge 0 \text{
on } \x \Rightarrow L(a/d) \ge 0,
\]
there exists a Borel measure $\mu$ on $\x$ such that
\[
\forall a \in A, \forall d \in D, \ L(a/d)=\int_\x \hat{a}/\hat{d}
\ d \mu.
\]
\end{thm}

The idea is to follow the proof of \cite[Theorem 3.2.2]{mar} (or \cite[Theorem 3.1]{mar3})
and use \cite[Theorem 3.2]{v} instead of Riesz's Theorem. We start with a few remarks:

\subsection*{R1} Assumption (1) implies that $\x$ is locally compact.

\subsection*{R2} For every $d \in D$, the following assertions are equivalent:
   \begin{enumerate}
   \item[(a)] $\x \setminus \z(\hat{d})$ is dense in $\x$,
   \item[(b)] $\z(\hat{d})$ has empty interior,
   \item[(c)] for every $f \in \mathcal{C}(\x,\RR)$, if $f \hat{d}^2 \ge 0$ on $\x$ then $f \ge 0$ on $\x$,
   \item[(d)] for every $f \in \mathcal{C}(\x,\RR)$, if $f \hat{d} \equiv 0$ on $\x$ then $f \equiv 0$ on $\x$.
   \end{enumerate}
Every locally compact Hausdorff space is completely regular, which
means that points can be separated from closed sets by continuous
functions (Urysohn's Lemma, see \cite[2.12]{rud}).  So we have
that (d) $\Rightarrow$ (b). Implications (a) $\Leftrightarrow$ (b)
$\Rightarrow$ (c) $\Rightarrow$ (d) are clear.

\subsection*{R3} Assumption (2) implies that $0 \not\in \hat{D}$, hence $0 \not\in D$.
Therefore, $D^{-1}A$ and $\hat{D}^{-1} \mathcal{C}(\x,\RR)$ are
nontrivial $\RR$-algebras.

\subsection*{R4} The  homomorphism $\ \hat{} \colon A \to \mathcal{C}(\x,\RR)$ extends uniquely
to a homomorphism $\ \hat{} \colon D^{-1}A \to \hat{D}^{-1}
\mathcal{C}(\x,\RR)$. Its image is $\widehat{D^{-1} A} =
\hat{D}^{-1} \hat{A}$.

\subsection*{R5} $\bar{L} \colon \hat{D}^{-1} \hat{A} \to \RR$ defined by
$\bar{L} (\hat{a}/\hat{d}) = L(a/d)$ is well-defined. This follows
from the positivity assumption on $L$. (See Claim 1 in the proof
of \cite[Theorem 3.2.2]{mar}.)

\subsection*{R6} We can assume that $\hat{d} \ge 0$ on $\x$ for every $d \in D$.
Namely, every element $f/\hat{d} \in \hat{D}^{-1}
\mathcal{C}(\x,\RR)$ can be written as
$f/\hat{d}=f\hat{d}/\widehat{d^2}$ and $\widehat{d^2} \ge 0$ on
$\x$. The positivity assumption on $L$ then simplifies to: $
\forall a \in A, \forall d \in D, \ \hat{a} \ge 0 \text{ on } \x
\Rightarrow L(a/d) \ge 0. $

\subsection*{R7}Let $B= \hat{D}^{-1} \mathcal{C}'(\x,\RR)$,
where $\mathcal{C}'(\x,\RR)= \{f \in \mathcal{C}(\x,\RR) \mid
\exists \, a \in A \colon |f| \le \hat{a}\text{ on } \x\}$. Note
that $B$ is a subalgebra of $\hat{D}^{-1} \mathcal{C}(\x,\RR)$ and
$B=\{f/\hat{d} \in \hat{D}^{-1} \mathcal{C}(\x,\RR) \mid \exists
\, a\in A, s \in D \colon |f| \hat{s} \le \hat{a}\hat{d} \text{ on
} \x\}$.  Namely, if  $|f| \hat{s} \le \hat{a}\hat{d}$ on $\x$,
then $f/\hat{d}= f\hat{s}/\widehat{ds}$ and  $f \hat{s} \in
\mathcal{C}'(\x,\RR)$. The other inclusion is clear.

\subsection*{R8} The algebra $B$ obviously contains
both $\hat{D}^{-1} \hat{A}$ and  $\hat{D}^{-1}
\mathcal{C}_c(\x,\RR)$. Here, $\mathcal{C}_c(\x,\RR)$ denotes the
algebra of continuous functions with compact support.

\begin{lemma}
\label{lem1} $\bar{L}$ extends to a linear functional $\bar{L}$ on
$B= \hat{D}^{-1} \mathcal{C}'(\x,\RR)$ such that
$\bar{L}(f/\hat{d})\ge 0$ for every $d \in D$ and every
nonnegative $f \in \mathcal{C}'(\x,\RR)$.
\end{lemma}

\begin{proof}
The proof is similar as the proof of Claim 2 in the proof of
\cite[Theorem 3.2.2]{mar}.) This result is also a special case of
a more general extension theorem for positive functionals, see
\cite[Theorem 2.6.2, p. 69]{ak}.
\end{proof}

\begin{lemma}[Rational Riesz's Theorem]
\label{lem2} For every linear functional $\bar{L} \colon
\hat{D}^{-1} \mathcal{C}_c(\x,\RR) \to \RR$ such that
$\bar{L}(f/\hat{d}) \ge 0$ for every nonnegative $f \in
\mathcal{C}_c(\x,\RR)$ there exists a unique Borel measure $\mu$
on $\x$ such that $\bar{L}(f/\hat{d}) = \int_\x f/\hat{d} \ d \mu$
for every $f/\hat{d} \in \hat{D}^{-1} \mathcal{C}_c(\x,\RR)$.
\end{lemma}

\begin{proof} (Based on the proof of \cite[Theorem 3.2]{v}.)
For every $s \in D$ define a functional $\bar{L}_s \colon
\mathcal{C}_c(\x,\RR) \to \RR$ by $\bar{L}_s
(f)=\bar{L}(f/\hat{s})$. Since $\bar{L}_s(f) \ge 0$ for every $f
\ge 0$, there exists by Riesz's Theorem, see \cite[Theorem
2.14]{rud},  a unique regular Borel measure $\mu_s$ on $\x$ such
that $\bar{L}_s(f) = \int_{\x} f \ d \mu_s$ for every $f \in
\mathcal{C}_c(\x,\RR)$. Write $\mu=\mu_1$ and note that
\[
\int_{\x} f \, d \mu = \int_{\x} f \hat{s} \ d \mu_s
\]
for every $f \in \mathcal{C}_c(\x,\RR)$. By \cite[Theorem
1.29]{rud}, $\hat{s} \ d \mu_s$ is a Borel measure on $\x$. To
prove that $\hat{s} \ d \mu_s$ is regular, we apply \cite[Theorem
2.18]{rud}. (We need here our assumption (3) and continuity of
$\hat{s}$.) By the uniqueness part of Riesz's Theorem,
$\mu=\hat{s} \ d \mu_s$.

Now, we will prove that $\mu_s(\z(\hat{s})) = 0$. By a version of
Urysohn's Lemma (see \cite[Theorem 2.12]{rud}), there exists for
each integer $k$ a function $u_k \in \mathcal{C}_c(\x,\RR)$ such
that $0 \le u_k \le 1$ on $\x$ and $u_k=1$ on $\z(\hat{s}) \cap
\x_k$. For all integers $i,k$, the function $f_{k,i}=u_k \min\{1,
1/i \hat{s}\}$ belongs to $\mathcal{C}_c(\x,\RR)$, $0 \le f_{k,i}
\le 1$ on $\x$ and $f_{k,i}=1$ on $\z(\hat{s}) \cap \x_k$. It
follows that
\[
\mu_s(\z(\hat{s}) \cap \x_k) \le \int_{\x} f_{k,i} \ d \mu_s =
\]
\[
=\bar{L}_s(f_{k,i}) = \bar{L}(f_{k,i}/\hat{s}) \le (1/i)
\bar{L}(u_k/\hat{s}^2).
\]
Sending $i \to \infty$, we get $\mu_s(\z(\hat{s}) \cap \x_k) = 0$
for every $k$, hence $\mu_s(\z(\hat{s})) = 0$. In particular,
$\mu(\z(\hat{s}))=0$ (which is also a consequence of $\mu=\hat{s}
\ d \mu_s$). Finally, for every $f \in  \mathcal{C}_c(\x,\RR)$ and
every $s \in D$ we have that
\[
\int_{\x} f/\hat{s} \ d \mu = \int_{\x \setminus \z(\hat{s})}
f/\hat{s} \ d \mu = \int_{\x \setminus \z(\hat{s})} (f/\hat{s})
\hat{s} \ d \mu_s =
\]
\[
=\int_{\x \setminus \z(\hat{s})} f \ d \mu_s = \int_{\x} f \ d
\mu_s = \bar{L}_s(f) = \bar{L}( f/\hat{s}).
\]

\end{proof}

\begin{lemma}
\label{lem3} For every nonnegative function $f \in
\mathcal{C}'(\x,\RR)$, there exists a monotonically increasing
sequence of nonnegative functions $f_i \in \mathcal{C}_c(\x,\RR)$
such that $\frac{(f+\hat{q})^2}{i} \ge f-f_i \ge 0$ on $\x$ for
every $i$.
\end{lemma}

This is Claim 3 in the proof of \cite[Theorem 3.2.2]{mar}. No
changes are needed.

\renewcommand{\proofname}{Proof of the Theorem}
\begin{proof}
Let $\bar{L}$ be a linear functional on $\hat{D}^{-1} \hat{A}$
such that $\bar{L}(f/\hat{d}) \ge 0$ for every nonnegative $f \in
\hat{A}$ and every $d\in D$. By Lemma \ref{lem1}, we can extend
$\bar{L}$ to a linear functional $\bar{L}$ on $\hat{D}^{-1}
\mathcal{C}'(\x,\RR)$ such that $\bar{L}(f/\hat{d}) \ge 0$ for
every nonnegative $f \in \mathcal{C}'(\x,\RR)$. Lemma \ref{lem2}
gives us a Borel measure $\mu$ on $\x$ such that
$\bar{L}(f/\hat{d})=\int_{\x} f/\hat{d} \ d \mu$ for every $f \in
\mathcal{C}_c(\x,\RR)$. Lemma \ref{lem3} implies that
$\bar{L}(f/\hat{d})=\int_{\x} f/\hat{d} \ d \mu$ for every $f \in
\mathcal{C}'(\x,\RR)$. Namely,
\[
\int_{\x} f/\hat{d} \ d \mu=\lim\limits_{i \to \infty}\int_{\x}
f_i/\hat{d} \ d \mu =\lim\limits_{i \to \infty}
\bar{L}(f_i/\hat{d}) = \bar{L}(f/\hat{d}),
\]
by the Monotone Convergence Theorem, the inequality
$\frac{(f+\hat{q})^2}{i} \ge f-f_i \ge 0$ and the positivity of
$\bar{L}$. In particular
$L(a/d)=\bar{L}(\hat{a}/\hat{d})=\int_{\x} \hat{a}/\hat{d} \ d
\mu$ for every $a \in \A$ and $d \in D$.
\end{proof}
\renewcommand{\proofname}{Proof}

\section{Preorderings and the topology of finitely open sets}
\label{variant}

\subsection*{Definition}
Let $E$ be an $\RR$-vector space and $O$ a subset of $E$. We say
that $O$ is \text{open} in the \textit{topology of finitely open
sets} if for every finite-dimensional vector subspace $U$ of $E$,
$O \cap U$ is open in the usual topology of $U$. Write
$\mathcal{T}_{\fin}$ for the set of all open sets.

\bigskip

We recall the following results from \cite{b}.

\begin{enumerate}
\item[(F1)] If $E$ has countable dimension, then the topology
$\mathcal{T}_{\fin}$ is equal to the topology
$\mathcal{T}_{\omega}$. If $E$ does not have countable dimension,
then the topology $\mathcal{T}_{\fin}$ is strictly finer (=has
more open sets) than the topology $\mathcal{T}_{\omega}$. In
particular, $\mathcal{T}_{\fin}$ is not a locally convex topology
in this case. \item[(F2)] If $E$ does not have countable
dimension, then the addition is not continuous in the topology
$\mathcal{T}_{\fin}$. \item[(F3)] Let $E$ be an $\RR$-vector space
and $\mathcal{D}$ the set of all countable-dimensional subspaces
of $E$. Then, for every subset $A$ of $E$
\[
\overline{A}=\bigcup_{D \in \mathcal{D}} \overline{A \cap D}.
\]
\end{enumerate}

The following should be well-known but we were unable to find a reference:

\begin{enumerate}
\setcounter{enumi}{3}
\item[(F4)] Any linear mapping between any two vector spaces
both having the topology of finitely open sets is continuous. In particular,
every functional on every vector space with the finitely open topology is continuous.
\end{enumerate}

\begin{proof}
Let $L \colon (E, \mathcal{T}_{\fin}) \to (E', \mathcal{T}_{\fin}')$ be
a linear mapping and $O \in \mathcal{T}_{\fin}'$. Pick any finite-dimensional
subspace $V$ of $E$. Clearly, the subspace $L(V)$ of $E'$ is also finite-dimensional.
By the definition of $\mathcal{T}_{\fin}'$, $O \cap L(V)$ is open in the usual
topology of $L(V)$. Every linear mapping between two finite-dimensional
vector spaces is continuous in the usual topologies. In particular
$L|_V \colon V \to L(V)$ is continuous in the usual topologies of $V$ and $L(V)$.
It follows, that the set $(L|_V)^{-1}(O \cap L(V))$ is open in the usual
topology of $V$. Note that $(L|_V)^{-1}(O \cap L(V))=L^{-1}(O) \cap V$,
hence $L^{-1}(O) \cap V$ is open in the usual topology of $V$.
Since $V$ was arbitrary, it follows by the definition of $\mathcal{T}_{\fin}$
that $L^{-1}(O)$ is open. Since $O$ was arbitrary, $L$ is indeed continuous.
\end{proof}

\begin{lemma}
\label{finclos}
For a preordering $T$ on a commutative $\RR$-algebra $E$, its
closure in $\mathcal{T}_{\fin}$ is also a preordering.
\end{lemma}

\begin{proof}
For any $f_1,f_2 \in \overline{T}$, there exist, by (F3), countable-dimensional
subspaces $D_1,D_2$ of $E$ such that $f_i \in \overline{T \cap D_i}$ for $i=1,2$.
The subspace $D=D_1+D_2$ is also countable-dimensional and
$\overline{T \cap D_i} \subseteq \overline{T \cap D}$ for $i=1,2$.
It follows that $f_1,f_2 \in \overline{T \cap D}$. Since $D$ is
countable-dimensional, $\overline{T \cap D}$ is closed under addition,
hence $f_1+f_2 \in \overline{T \cap D} \subseteq \overline{T}$.
The proof that $\overline{T}$ is closed under multiplication
follows again from the fact, that $u\mapsto ru$ is continuous for all $r\in E$.
(See the proof of Lemma \ref{l}).
\end{proof}

\begin{lemma}\label{polyrat}
Let $D\subseteq\RR[\underline{X}]\setminus\{0\}$ be a
multiplicative set containing $1$. Then a preordering $T$ in
$E=D^{-1}\RR[\underline{X}]$ is closed if and only if $T\cap
\RR[\underline{X}]$ is closed in $\RR[\underline{X}]$.
\end{lemma}

\begin{proof} One direction is obvious. So assume that $T \cap \RR[\underline{X}]$
is closed in $\RR[\underline{X}]$. Let $V$ be a finite-dimensional
subspace of $E$. We claim that $V \cap T$ is closed. Pick a
nonzero square $d\in D$ such that $dV \subseteq
\RR[\underline{X}]$. Then
\[
d(V \cap T) = dV \cap dT= dV \cap T.
\]
The first equality follows from the fact that the mapping $u
\mapsto du$ is bijective on $E$ and the second from $T=dT$ ($d$ is
an invertible square in $E$). Since $T \cap \RR[\underline{X}]$ is
closed and $dV\subseteq \RR[\underline{X}]$, $d(V \cap T)$ is
closed in $dV$. As the mapping $u \mapsto du$ is a homeomorphism
by (F4), it follows that the set $V \cap T$ is closed in $V$. As
$V$ was arbitrary, $T$ is closed in $E$.
\end{proof}

The next Lemmas are needed in the proof of the main theorem of
this section.  See \cite[Chapter 2]{bcr} for general notions such
as semialgebraic set, semialgebraic homeomorphism, semialgebraic
dimension etc.

\begin{lemma}\label{decomp} Let $K\subseteq\RR^n$ be a semialgebraic set. Then
there are finitely many semialgebraic subsets $C_1,\ldots,C_s$ of
$K$, such that \begin{itemize} \item[(i)] $K=C_1\cup \ldots \cup
C_s$, \item[(ii)] each $C_i$ is semialgebraically homeomorphic to some
$\RR^{d_i}$, \item[(iii)] each $C_i$ has irreducible Zariski
closure.
\end{itemize}\end{lemma}
\begin{proof}
The proof is by induction on the semialgebraic dimension of $K$.
If $\dim(K)=0$, the result is clear, as $K$ is a finite union of
points. So assume $\dim(K)=m$ and the result is true for all
semialgebraic sets of smaller dimension. Let $V$ be the Zariski
closure of $K$ and $V_1,\ldots, V_r$ its irreducible components.
We consider with out loss of generality only  $K_1:= K\cap V_1$.
The usual semialgebraic cell decomposition allows us to write
$$K_1= \tilde{C}_1\cup\ldots\cup \tilde{C}_k,$$ where each
$\tilde{C}_j$ is a semialgebraic set and semialgebraically
homeomorphic to some $\RR^{d_j}$. Now if a  $\tilde{C}_j$  has
the same dimension as $V_1$, its Zariski closure equals $V_1$, due
to the irreducibility of $V_1$. So we use all these $C_j$  for our
desired decomposition of $K$. The union of all lower dimensional
$C_j$ is a semialgebraic set of strictly smaller dimension than
$m$, so we can apply the induction hypothesis to it. This yields
the result.
\end{proof}

\begin{lemma}
\label{dk} For every semialgebraic set $K$ in $\RR^n$ and every
nonempty multiplicative subset $D$ of $\RR[\underline{X}]$, there
exists an element $d_K \in D$ such that $\overline{K \setminus
\z(d_K)} \subseteq \overline{K \setminus \z(d)}$ for every $d \in
D$.
\end{lemma}

\begin{proof}
Write $K=C_1\cup\ldots\cup C_s$ with the properties from Lemma
\ref{decomp}. For each $i \in \{1,\ldots,s\}$ and $d \in D$ there
are two possibilities. First, $\dim(C_i \cap \z(d))<\dim(C_i)$.
Since semialgebraic homeomorphisms respect dimensions
(\cite[Theorem 2.8.8]{bcr}), we see that $C_i \cap \z(d)$ has
(relative) empty interior in $C_i$, so $\overline{C_i \setminus
\z(d)}=\overline{C_i}.$ The second possibility is $\dim(C_i \cap
\z(d))=\dim(C_i).$ Since the Zariski closure of $C_i$ is
irreducible, it equals the Zariski closure of $C_i \cap \z(d).$
But this means $d$ must vanish on the whole of $C_i$. So
$\overline{C_i \setminus \z(d)}=\emptyset.$ So we have proved that
for every $d \in D$,
\[
\overline{K \setminus \z(d)}= \bigcup_{i=1}^s \overline{C_i
\setminus \z(d)}= \bigcup_{i \in I(d)} \overline{C_i},
\]
where $I(d)=\{i \mid \dim(C_i \cap \z(d)) < \dim(C_i)\}.$ There
exist finitely many $d_1,\ldots,d_t \in D$ such that $\bigcap_{d
\in D} I(d) = \bigcap_{j=1}^t I(d_j).$ Set $d_K := \prod_{j=1}^t
d_j$ and note that $I(d_K) =\bigcap_{j=1}^t I(d_j).$ It follows
that for every $d \in D$,
\[
\overline{K \setminus \z(d_K)}=\bigcup_{i \in I(d_K)}
\overline{C_i} \subseteq \bigcup_{i \in I(d)}
\overline{C_i}=\overline{K \setminus \z(d)}.
\]
\end{proof}

\begin{lemma}
\label{clospos} For any semialgebraic set $K\subseteq\RR^n$, the
preordering
\[
\pos^{E}(K)= \{R \in D^{-1}\RR[\underline{X}] \mid R \succeq 0
\mbox{ on } K\}
\]
is closed in $\mathcal{T}_{\fin}$.
\end{lemma}

\begin{proof}
Lemma \ref{dk} implies that the set $\x= \overline{K \setminus \z(d_K)}$
satisfies the assumptions of Lemma \ref{posi}, hence
$\pos^E(\x) \cap \RR[\underline{X}] = \pos(\x)$.
The set $\pos(\x)$ is closed in $\RR[\underline{X}]$, as evaluations in
points are  linear and therefore continuous.
So $\pos^E(\x)$ is closed by Lemma \ref{polyrat}.

We also have that
$\pos^E(\x)=\pos^E(K)$. Namely, if $R=\frac{f}{d}$ with $f \in
\RR[\underline{X}]$, and $d \in D$ such that $fd \in \pos(K
\setminus \z(d_K))$, then also $R=\frac{fd_K}{dd_K}$ with
$(fd_K)(dd_K) \in \pos(K)$. The opposite inclusion is trivial.
\end{proof}

The following main theorem of this section is an analogue of
Theorem \ref{main} above.

\begin{thm}
\label{main2}
Let $D\subseteq\RR[\underline{X}]\setminus\{0\}$ be a
multiplicative set containing $1$ and
$S\subseteq\RR[\underline{X}]$ finite.  Suppose there is $p\in D$
with $p\geq 1$ on $K_S$ and $kp-\sum_{i=1}^n X_i^2\geq 0$ on $K_S$
for some $k\geq 1$. Then in the topology $\mathcal{T}_{\fin}$ on
$D^{-1}\RR[\underline{X}]$ we have
$$\overline{T_S^E}=\pos^E(K_S).$$
\end{thm}

\begin{proof} First take  $R\in \pos^E(K_S)$, so $R=\frac{f}{d}$ for some $f\in\RR[\underline{X}],d\in D$, with
$fd \ge 0$ on $K_S$. We apply \cite[Theorem 5.1]{sch} and find
$fd$ lying in the closure of $T_S^E$ in a finite dimensional
subspace of $D^{-1}\RR[\underline{X}]$. So $fd$ and therefore $R$
belong to $\overline{T_S^E}.$

Since $T_S^E \subseteq \pos^E(K_S) \subseteq \overline{T_S^E}$, it
remains to show that $\pos^E(K_S)$ is closed in
$\mathcal{T}_{\fin}$. This is Lemma \ref{clospos}.
\end{proof}

\begin{remark}
We can also deduce Theorem \ref{main2} from  Theorem \ref{main}
and Lemma \ref{clospos}. Namely,
\[
\bigcup_{d \in D} \pos^{\RR[\underline{X}]_d}(K_S) \subseteq
\bigcup_{d \in D} \pos^{\RR[\underline{X}]_{pd}}(K_S) \subseteq
\bigcup_{d \in D} \overline{T_S^{\RR[\underline{X}]_{pd}}} \subseteq
\bigcup_{d \in D} \overline{T_S^{\RR[\underline{X}]_d}},
\]
by application of Theorem \ref{main} to the algebra
$\RR[\underline{X}]_{pd}$; we use that $\mathcal{T}_{\omega}$ and
$\mathcal{T}_{\fin}$ coincide in the countable dimensional case.
Now, using Lemma \ref{clospos},
\[
\overline{T_S^E} \subseteq
\pos^E(K_S) \subseteq
\bigcup_{d \in D} \pos^{\RR[\underline{X}]_d}(K_S) \subseteq
\bigcup_{d \in D} \overline{T_S^{\RR[\underline{X}]_d}} \subseteq
\overline{T_S^E}.
\]
\end{remark}

The following is a counterpart to Corollary \ref{rat}.
\begin{cor}
$\sum\RR(\underline{X})^2$ is closed in $\RR(\underline{X})$ with
respect to $\mathcal{T}_{\fin}$.
\end{cor}
\begin{proof}
Every globally nonnegative rational function is a sum of squares
of rational functions.
\end{proof}

\subsection*{Example} In $\RR(\underline{X};\emptyset)$,
the sums of squares are not closed. It is a long known fact that
there are nonnegative polynomials which can not be written as a
sum of squares of rational functions without poles. See for
example \cite[Lemma 1.1]{s} or \cite[Paragraph 6]{r2}. But as we
have seen, these polynomials belong to the closure of the sums of
squares in $\RR(\underline{X};\emptyset)$.

\section{Extension to finitely generated $\RR$-algebras}

The aim of this section is to extend Theorems \ref{main}, \ref{mainmom} and \ref{main2}
from polynomial rings to finitely generated $\RR$-algebras.

Let $A$ be a finitely generated $\RR$ algebra. A \textit{character} of $A$
is a unital $\RR$-algebra homomorphisms from $A$ to $\RR$. Write
$V(A)$ for the set of all characters of $A$. Elements from $A$
define functions on $V(A)$ by $a(\chi):=\chi(a).$ We equip $V(A)$
with the coarsest topology making all these functions continuous.
One can embed $V(A)$ into some $\RR^n$, by choosing generators
$x_1,\ldots,x_n$ of $A$ and sending $\chi\mapsto
(\chi(x_1),\ldots, \chi(x_n))$. In that case, the topology on
$V(A)$ coincides with the usual topology from $\RR^n$ (see for
example \cite[Section 5.7]{mar}).

Let $D$ be a multiplicative subset of $A \setminus \{0\}$ which
contains $1$. There is a canonical homomorphism $\iota\colon
A\rightarrow F=D^{-1}A$, which however is not necessarily
one-to-one. For every $d \in A$ write $\z_A(d)=\{\chi \in V(A)
\mid \chi(d) =0\}$. Note that the set $V(A) \setminus \bigcup_{d \in D} \z_A(d)$
consists of all characters  of $A$ which can be extended (via $\iota$) to a
character of $F$.

Let $S$ be a finite subset of $A$. Write $K_S^A=\{\chi \in V(A)
\mid \chi(S) \ge 0\}$. Let $T_S^F$ be the preordering in
$F$ generated by the set $\iota(S)$.

For every subset $\x$ of $V(A)$ write $\pos(\x)$ for the set of all
$a \in A$ such that $\chi(a) \ge 0$ for every $\chi \in \x$ and write
$\pos^F(\x)$ for the set of all elements $R \in F$ which have
a representation $R=\frac{a}{d}$ with $a \in A$, $d \in D$ and $ad \in \pos(\x)$.

\begin{thm}
\label{mainfg}
Assume that $A,D,S,F$ are as above. Also assume that there is an element $p \in D$
such that $p-1 \in \pos(K_S^A)$ and $kp-\sum_{i=1}^n x_i^2  \in \pos(K_S^A)$
for some integer $k \ge 1$ and generators $x_1,\ldots,x_n$ of $A$. Then
\begin{enumerate}
\item The closure of $T_S^F$ in $(F,\mathcal{T}_{\omega})$ is
$\pos^F(K_S^A \setminus \bigcup_{d \in D} \z_A(d))$.
\item The closure of $T_S^F$ in $(F,\mathcal{T}_{\fin})$ is $\pos^F(K_S^A)$.
\item For every linear functional $L$ on $F$ such that $L(T_S^F) \ge 0$, there
exists a measure $\mu$ on $\overline{K_S^A \setminus \bigcup_{d \in D}\z_A(d)}$ such that
\[
L\left(\frac{f}{d}\right) = \int \frac{f}{d} \ d \mu
\quad
\text{for every}
\quad
\frac{f}{d} \in F.
\]
\end{enumerate}
\end{thm}

Note that the set $\pos^F(K_S^A \setminus \bigcup_{d \in D} \z_A(d))$
consists of all $R \in F$ such that $\tau(R) \ge 0$ for every character $\tau$
of $F$ which satisfies $\tau(\iota(S)) \ge 0$.

\begin{proof}
Let $\pi \colon \RR[\underline{X}] \to A$ be the unital $\RR$-algebra
homomorphism defined by $X_i \to x_i$ for $i=1,\ldots,n$ and
$\pi^\ast \colon V(A) \to \RR^n$ the corresponding embedding
defined by $\chi \mapsto (\chi(x_1),\ldots,\chi(x_n))$.
If $f_1,\ldots,f_t$ are generators of the ideal $I=\ker \pi$,
then $\pi^\ast(V(A))=K_{\{\pm f_1,\ldots,\pm f_t\}}$.
Similarly, for a given set $S=\{g_1,\ldots,g_m\} \subseteq A$ take
elements $\tilde{g}_i\in\RR[\underline{X}]$ such that
$\pi(\tilde{g}_i)=g_i$ and note that
\[
\pi^\ast(K^A_S)=K_{\widetilde{S}}
\quad \text{ where } \quad
\widetilde{S}:=\left\{\tilde{g}_1,\ldots,\tilde{g}_m,\pm\tilde{f}_1,\ldots,\pm\tilde{f}_t\right\}.
\]
Let $\widetilde{D}:=\pi^{-1}(D)$ denote the
localization $\widetilde{D}^{-1}\RR[\underline{X}]$ by $E$. $\pi$
extends uniquely to a homomorphism $\tilde{\pi}\colon E\rightarrow
F$, making the following diagram commutative:
\[
\xymatrix{ E \ar@{->}^{\tilde{\pi}}[r] & F \\ \RR[\underline{X}]
\ar@{->}^{\pi}[r] \ar@{->}[u] & A \ar@{->}^{\iota}[u] }
\]

The preorderings $T_{\widetilde{S}}\subseteq\RR[\underline{X}]$
and $T_{\widetilde{S}}^E\subseteq E$ do not depend
on the specific choice of the $\tilde{f}_i,\tilde{g}_i$
but only on $I$ and $S$. We have
\[
\tilde{\pi}(T_{\widetilde{S}}^E)=T_S^F\tag{*}.
\]
Now suppose $D$ contains an element $p$ with the properties
required in Theorem \ref{mainfg}. Any preimage
$\tilde{p}$ of $p$ under $\pi$ will have the corresponding
property with respect to $K_{\widetilde{S}}.$

 For any linear functional $L$ on $F$
with $L(T_S^F)\geq 0$, the functional
$\widetilde{L}:=L\circ\tilde{\pi}$ fulfills
$\widetilde{L}(T_{\widetilde{S}}^E)\geq 0,$ and is therefore
integration on
$\overline{K_{\widetilde{S}}\setminus\bigcup_{\tilde{d}\in\widetilde{D}}
\z(\tilde{d})}$ by Theorem \ref{mainmom}.
It follows that $L$ is an integration on
\[
\overline{K_S^A\setminus\bigcup_{d\in D}\z(d)})=
(\pi^\ast)^{-1}(\overline{K_{\widetilde{S}}\setminus\bigcup_{\tilde{d}\in\widetilde{D}}
\z(\tilde{d})}).
\]
So we have proved assertion (3) of Theorem \ref{mainfg}.
Assertions (1) and (2) of Theorem \ref{mainfg} follow from Theorems \ref{main}
and \ref{main2}, observations
\[
\tilde{\pi}(\pos^E(K_{\widetilde{S}} \setminus
\bigcup_{\tilde{d} \in \widetilde{D}} \z(\tilde{d})))=
\pos^F(K_S^A \setminus \bigcup_{d \in D} \z_A(d)),
\]
\[
\tilde{\pi}(\pos^E(K_{\widetilde{S}}))=\pos^F(K_S^A)
\]
and (*) and from the following Lemma.
\end{proof}

\begin{lemma}
\label{quot2} Let $E$ and $F$ be $\RR$-vector spaces, $\tilde{\pi} \colon
E \to F$ a linear mapping which is onto and $C$ a convex cone in
$E$ which contains $\ker \tilde{\pi}$. Then
\[
\tilde{\pi}(\overline{C})=\overline{\tilde{\pi}(C)},
\]
where either $E,F$ are both equipped with $\mathcal{T}_{\omega}$ or
both with $\mathcal{T}_{\fin}$.
\end{lemma}

\begin{proof}
Suppose that $E,F$ are  equipped with $\mathcal{T}_{\omega}$.
The inlusion $\tilde{\pi}(\overline{C}) \subseteq \overline{\tilde{\pi}(C)}$ follows from the fact
that $\tilde{\pi}$ is continuous in $\mathcal{T}_{\omega}$. To prove the
opposite inclusion, pick $\tilde{\pi}(e) \in \overline{\tilde{\pi}(C)}$. Then $L(\tilde{\pi}(e)) \ge 0$
for every linear functional $L$ on $F$ such that $L(\tilde{\pi}(C)) \ge 0$.
Every linear functional $L'$ on $E$ such that $L'(C) \ge 0$ factors through $\tilde{\pi}$
because $\ker \tilde{\pi} \subseteq C$. It follows that $L'(e) \ge 0$ for every
linear functional $L'$ on $E$ such that $L'(C) \ge 0$. So, $e \in \overline{C}$
which implies that $\tilde{\pi}(e) \in \tilde{\pi}(\overline{C})$.

Suppose now that $E,F$ are  equipped with $\mathcal{T}_{\fin}$.
 First note that by the same argument as in Lemma \ref{finclos},
the closure of a convex cone with respect to $\mathcal{T}_{\fin}$ is again
a convex cone. Now the formula $\tilde{\pi}(\overline{C}) \subseteq \overline{\tilde{\pi}(C)}$
follows from the fact that $\tilde{\pi}$ is continuous. To prove the opposite
inclusion it suffices to show that $\tilde{\pi}(C)$ is closed for every
closed cone $C$ containing $\ker \tilde{\pi}$. Suppose that $C$ is such a
cone and pick a finite-dimensional subspace $W$ of $F$. Let $V$ be
a finite-dimensional subspace of $E$ such that $\tilde{\pi}(V)=W$. By
assumption, $C \cap V$ is closed in $V$ in the Euclidean topology,
hence also in the finest locally convex topology. Since $C \cap V$
is a closed cone in $V$ which contains $V \cap \ker \tilde{\pi} = \ker
\tilde{\pi}|_V$, it follows by the first paragraph (applied to $\tilde{\pi}|_V \colon
V \to W$ instead of  $\tilde{\pi} \colon E \to F$) that $\tilde{\pi}(C \cap V)$ is
closed in $W=\tilde{\pi}(V)$. It remains to show that $\tilde{\pi}(C \cap
V)=\tilde{\pi}(C) \cap \tilde{\pi}(V)$. Pick $\tilde{\pi}(e) \in \tilde{\pi}(C) \cap \tilde{\pi}(V)$. We
have that $e=c+i=v+j$ for some $c \in C$, $v \in V$ and $i,j \in
\ker \tilde{\pi}$. Since $\ker \tilde{\pi} \subseteq C$, $c+i-j=v$ belongs to $C
\cap V$. It follows that $\tilde{\pi}(e)=\tilde{\pi}(v)$ belongs to $\tilde{\pi}(C \cap
V)$. The opposite inclusion is clear.
\end{proof}

It would be interesting to know whether Theorem \ref{mainfg} also holds
for the algebra $A=\mathcal{O}(\RR^n)$ of analytic functions on $\RR^n$. In this case we can use
\cite[Theorem 2.4]{aab}  instead of  \cite[Theorem 5.1]{sch}, so the assumption
$p \in D$ for certain $p$ may not be necessary.


\begin{thebibliography}{999}
\bibitem{aab}  F. Acquistapace, C. Andradas, F. Broglia,
\textit{The strict Positivstellensatz for global analytic functions and the moment problem for semianalytic sets},
Math. Ann.  \textbf{316}  (2000), pp. 609-616.
MR1758445 (2001g:14087)

\bibitem{ak} N. I, Akhiezer,
\textit{The Classical Moment Problem and Some Related Questions in
Analysis}, Oliver \& Boyd, Edinburgh/London, 1965.
MR0184042 (32 \#1518)

\bibitem{b} T. M. Bisgaard,
\textit{The topology of finitely open sets is not a vector space topology}, 
Arch. Math. \textbf{60} (1993), pp. 546-552.
MR1216700 (94g:46013)  

\bibitem{bcr} J. Bochnak, M. Coste, M.-F. Roy, 
\textit{Real Algebraic Geometry}, 
Erg. Math. Grenzgeb. 36, Springer, Berlin (1998).
MR1659509 (2000a:14067)  

\bibitem{d} N. Bourbaki,
\textit{Topological vector spaces, Chapters 1-5}, English edition,
Springer Verlag, Masson 1987.
MR0910295 (88g:46002)  

\bibitem{bu} A. Bultheel, P. Gonz{\'a}lez-Vera, E. Hendriksen, O.
Nj{\aa}stad, \textit{Orthogonal rational functions}, Cambridge
Monographs on Applied and Computational Mathematics vol. 5,
Cambridge University Press, Cambridge, 1999.
MR1676258 (2000c:33001)  

\bibitem{bu2} A. Bultheel, P. Gonz{\'a}lez-Vera, E. Hendriksen, O. Nj{\aa}stad,
\textit{Elements of a theory of orthogonal rational functions},
Rev. Acad. Canaria Cienc. \textbf{11} (1999), pp. 127-152.
MR1780786 (2001i:42041)  

\bibitem{jdc} J. D. Chandler, Jr.,
\textit{Rational moment problems for compact sets}, 
J. Approximation Theory \textbf{79} (1994), pp. 72-88.
MR1294321 (95h:44019)   

\bibitem{km} S. Kuhlmann, M. Marshall,
\textit{Positivity, sums of squares and the multidimensional moment problem},
Trans. Amer. Math. Soc. \textbf{354} (2002), pp. 4285-4301.
MR1926876 (2003j:14078)  

\bibitem{kp} S. G. Krantz, H. R. Parks,
\textit{Geometric Integration Theory}, \\
http://www.math.wustl.edu/\~{}sk/books/root.pdf.

\bibitem{mar} M. Marshall,
\textit{Positive polynomials and sums of squares}, 
AMS Surveys and Monographs vol. 146, AMS, Providence, 2008.
MR2383959  

\bibitem{mar2} M. Marshall,
\textit{Extending the Archimedean Positivstellensatz to the non-compact case},
Canad. Math. Bull.  \textbf{44}  (2001), pp.  223-230.
MR1827856 (2002b:14073)  

\bibitem{mar3} M. Marshall,
\textit{Approximating positive polynomials using sums of squares},
Canad. Math. Bull., \textbf{46} (2003), pp. 400-418.
MR1994866 (2004f:14084)  

\bibitem{palm} T. W. Palmer, \textit{Banach algebras and the general theory of $*$-algebras, Vol. 2},
Encyclopedia of Mathematics and its Applications, 79. Cambridge University Press, Cambridge, 2001. pp. i--xii and 795--1617. 
MR1819503 (2002e:46002)  

\bibitem{pd} A. Prestel, C. N. Delzell,
\textit{Positive polynomials}, Springer, Berlin (2001).
MR1829790 (2002k:13044)  

\bibitem{pv} M. Putinar, F. -H. Vasilescu,
\textit{Solving moment problems by dimensional extension},
Ann. of Math. (2) \textbf{149}  (1999), pp. 1087-1107.
MR1709313 (2001c:47023b)  

\bibitem{r} B. Reznick,
\textit{Uniform denominators in Hilbert's Seventeenth Problem},
Math. Z. \textbf{220} (1995), pp. 75-97.
MR1347159 (96e:11056)  

\bibitem{r2} B: Reznick,
\textit{Some concrete aspects of Hilbert's 17th problem}, in:
Real Algebraic Geometry and Ordered Structures, (C. N. Delzell, J.J.
Madden eds.) Cont. Math. vol. 253 (2000), pp. 251-272.
MR1747589 (2001i:11042)  

\bibitem{rud} W. Rudin,
\textit{Real and Complex Analysis}, 3rd ed., McGraw-Hill 1987.
MR0924157 (88k:00002)  

\bibitem{s} C. Scheiderer,
\textit{Sums of squares of  regular functions on real algebraic varieties},
Trans. Am. Math. Soc. \textbf{352}  (1999), pp. 1039-1069.
MR1675230 (2000j:14090)  

\bibitem{kon} K. Schm\"udgen,
\textit{The K-moment problem for compact semi-algebraic sets},
Math. Ann. \textbf{289} (1991), pp. 203-206.
MR1092173 (92b:44011)  

\bibitem{sch} M. Schweighofer,
\textit{Iterated Rings of Bounded Elements and Generalizations of Schm\"udgen's Positivstellensatz},
J. Reine Angew. Math. \textbf{554} (2003), pp. 19-45. 
MR1952167 (2004b:13028)  
(unpublished Erratum available at \\
http://perso.univ-rennes1.fr/markus.schweighofer/publications/irobeerr.pdf)

\bibitem{ss} J. Stochel, F. H. Szafraniec, 
Algebraic operators and moments on algebraic sets,
Portugal. Math.  \textbf{51}  (1994),  25-45. MR1281954 (95d:47013)  

\bibitem{v} F. -H. Vasilescu,
\textit{Spaces of fractions and positive functionals},
Math. Scand. \textbf{96} (2005), 257-279.
MR2153414 (2007b:47040)  

\bibitem{w} J. H. Williamson,
\textit{On Topologizing the field $\CC(t)$},
Proc. Amer. Math. Soc.  \textbf{5} (1954), pp. 729-734.
MR0063574 (16,145f)  

\end{thebibliography}
\end{document}